\documentclass[12pt]{amsart}
\usepackage{amsfonts,amssymb,amsmath,amsthm}
\usepackage{enumerate}

\newtheorem{thrm}{Theorem}[section]
\newtheorem{lem}[thrm]{Lemma}
\newtheorem{prop}[thrm]{Proposition}

\newtheorem{definition}[thrm]{Definition}

\numberwithin{equation}{section}

\newcommand{\Z}{\mathbb{Z}}
\newcommand{\F}{\mathbb{F}}

\newcommand{\vr}{\vspace*{0.5cm}}

\author{Samir Assuena and C\'esar Polcino Milies }

\address{ Centro Universit\'ario da FEI, Av. Humberto de Alencar Castelo Branco 3972, S\~ao Bernardo do Campo, CEP:09850-901, Brazil and Centro Universit\'ario do Instituto Mau\'a de Tecnologia-I.M.T., 
Pra\c ca Mau\'a 1, S\~ao Caetano do Sul, SP, CEP: 09580-900,  Brazil}

\email{samir.assuena@fei.edu.br and samir.assuena@maua.br}

\address{Instituto de Matem\'{a}tica e Estat\'{i}stica, Universidade de S\~{a}o Paulo, Caixa Postal 66.281, CEP 05314-970, Brazil and CMCC, Universidade Federal do ABC,  Av. dos Estados, 5001, Santo Andr�, SP, CEP 09210-580, Brazil}  

\email{polcino@ime.usp.br and polcino@ufabc.edu.br}
\thanks{Research supported by CNPq Proc. 300243/79-0(RN), 579650/2008-1 and 141396/2009-1  }

\begin{document}

\title{Good codes from dihedral groups}

\begin{abstract}
We give a simple construction of codes from left ideals in  group algebras of certain dihedral groups and give an example to show  that they can produce codes with weights equal to those of the best known codes of the same length.
\end{abstract}
\maketitle

\section{Introduction}

Throughout this paper, \textit{G} will always denote a finite group and $\F$ a field such that $char(\F) \nmid |G|$. 

We recall that a group $G$ is   \textit{metacyclic} if $G$ contains a cyclic normal subgroup $H=\left\langle a\right\rangle$ such that the factor group $G/H=\left\langle bH\right\rangle$ is also cyclic.  
 
  Then:

\begin{center}
$G=\left\langle a,b \mid a^{m}=1, b^{n}=a^{s}, bab^{-1}=a^{i} \right\rangle$
\end{center}
 and the integers \textit{m,n,s, i} are such that

\begin{center}
$s \mid m, \hspace*{0.5cm} m \mid s(i-1),\hspace*{0.5cm} i<m, \hspace*{0.5cm} {\rm{gcd}}(i,m)=1.$
\end{center}

In the special case when $i=-1$ and $n=2$ we obtain the well-known family of {\em Dihedral groups}.

\vr

A \textit{group code} over a field $\F$ is any ideal $I$  of the group algebra $\F G$ of a finite   group $G$. A code is said to be metacyclic,   abelian, or dihedral in case the given group $G$ is of that kind of groups.

If $I$ is two-sided, then it is called a \textit{central} code. A  \textit{minimal central  code} is an ideal $I$ which is minimal in the set of all (two-sided) ideals of $\F G$. 

The \textit{Hamming distance} between two elements  $\alpha= \sum_{g \in G}\alpha_{g}g$ and $\beta= \sum_{g \in G}\beta_{g}g$ in $\F G$ is  

\begin{center}
$d(\alpha, \beta)=\mid \{g \mid \alpha_{g} \neq \beta_{g}, \,\, g \in G\} \mid$
\end{center}

\noindent that is, the number of elements of the support or $\alpha - \beta$. The \textit{weight}  of an ideal $I$ is the number 

\begin{center}
$w(I)={\rm{min}}\{d(\alpha, \beta) \mid \alpha\neq \beta, \,\, \alpha, \beta \in I \}$.
\end{center}


The following result was proved in {\rm{\cite[Proposition 2.1]{FRP}}} when $H \triangleleft G$  but, actually, the proof does not require normality.

\begin{lem} \label{1.4}
Let G be a finite group and let $\F$ a field. Let $H, K$ be subgroups of G with $H\subset K$. Write $e=\widehat{H}-\widehat{K}$. Then: 

\begin{enumerate}
	\item dim$_{\F}(\F G)e=(G:H)-(G:K)$;
	\item $w((\F G)e)=2|H|$;
  \item If ${\mathcal{A}}$ is a transversal of $K$ in $G$ and $\tau$ is a transversal of $H$ in $K$ containing $1$, then the set

\begin{center}
$\{r(1-t)\widehat{H} \mid r \in {\mathcal{A}}, t \in \tau\setminus \{1\} \}$
\end{center}
is a basis of $(\F G)e$ over $\F$.
\end{enumerate}

\end{lem}

Sabin and Lomonaco showed in \cite{SL} that central metacyclic codes are not better than cyclic codes, in a sense that, for each central metacyclic code, there exists a cyclic code that is equivalent to it. We shall show that, in the particular case of Dihedral codes, there exist non central codes that are not equivalent to any abelian code and, in the last section, we exibit examples of Dihedral codes that have weights  equal to  the weights of the best known codes of the same dimension.

\section{Dihedral Codes of Length 2$p^{m}$}

 Throughout this section, $D$ will denote the dihedral group of order 2$p^{m}$ with presentation

\begin{center}
$D=\left\langle a,b \mid a^{p^{m}}=1=b^{2}, \,\, bab=a^{-1}\right\rangle$.
\end{center}

\noindent and $\F_{q}$ will denote a finite field with $q$ elements such that $gcd(2p^{m}, q)=1$. Also, in what follows we shall always assume that ${\mathcal{U}}(\Z_{p^{m}})=\left\langle \overline{q}\right\rangle$; i.e. that the multiplicative order of $q$ modulo $p^n$ is $\varphi(p^n) = p^{n-1}(p-1)$.

Under this hypothesis, it has been shown in \cite{RP} that, setting   $A=\left\langle a\right\rangle$, the structure of the cyclic group algebra  $\F A$ depends only on the subgroup structure of $A$. As an application, one can describe the set of primitive idempotents of  $\F D$ as follows. Let

\begin{center}
$A=H_{0}\supseteq H_{1}\supseteq \cdots \supseteq H_{m}=\{1\}$ 
\end{center}
 be the descending chain of all subgroups of $A$, i.e., $H_{j}=\langle a^{p^{j}}, \rangle \ 0\leq i \leq m$. Consider the idempotents 

\begin{center}
$e_{0}=\widehat{A} \hspace*{0.5cm} e \hspace*{0.5cm}  e_{j}=\widehat{H_{j}}-\widehat{H_{j-1}},  \hspace*{0.2cm} 1\leq j\leq m$
\end{center} 
and write $e_{0}=\widehat{\langle a \rangle}$ as the following sum of idempotents

\begin{center}
 $e_{11}=\left(\frac{1+b}{2}\right)e_{0} \hspace*{0.5cm} e \hspace*{0.5cm} e_{22}=\left(\frac{1-b}{2}\right)e_{0}$.
\end{center}

\begin{thrm} \label{1.6} {\rm{\cite[Theorem 3.3]{FRP}}}
Under the hypotheses above, the set of primitive central idempotents of $\F_{q}D$ is 

\begin{center}
$\{e_{11}, e_{22}\}\cup \{e_{j}, \,\, 1\leq j\leq m \}.$ 
\end{center}

\end{thrm}

  Sabin and Lomonaco, \cite{SL} introduced the following.

\begin{definition}
Let $G_{1}$ and $G_{2}$ be finite groups of the same order and let $\F$ be a field. Let $\F G_{1}$ and $\F G_{2}$ be the corresponding group algebras. A \emph{combinatorial equivalence} is a $\F$ vector space isomorphism $\phi: \, \F G_{1}\longrightarrow \F G_{2}$ induced by a bijection $\phi: \, G_{1}\longrightarrow  G_{2}$. Codes $C_{1} \subset \F G_{1}$ and $C_{2} \subset \F G_{2}$ are \emph{combinatorially equivalent} if there exists a combinatorial equivalence $\phi: \, \F G_{1}\longrightarrow \F G_{2}$ such that $\phi(C_{1})=C_{2}$.
\end{definition}

Clearly, combinatorially equivalences are Hamming isometries and, thus, combinatorially equivalent codes have the same weights and dimensions.

\vr

Let $G$ be a metacyclic group with  presentation

\begin{center}
$G=\left\langle a,b \mid a^{m}=1=b^{n}, \,\ bab^{-1}=a^{i}\right\rangle$.
\end{center}
The result we quote below shows that central metacyclic codes are no better that abelian codes.

\begin{thrm} \label{1.8} {\rm{\cite [Theorem 1]{SL}}}
Let $C_{m}$ and $C_{n}$ be cyclic groups of order $m$ and $n$, with generators $\tilde{a}$ and $\tilde{b}$ respectively and let $\phi: \, G\longrightarrow C_{m}\times C_{n} $ be the bijection given by $\phi(a^{i}b^{j})=\tilde{a}^{i}\tilde{b}^{j}$ .
 If $C$ is a code generated by a central idempotent $e$ in $\F G$, then $\phi(C)$ is the  ideal of $\F [C_{m}\times C_{n}]$generated by  $\phi(e)$, an idempotent of $\F [C_{m}\times C_{n} ]$. Consequently, every central metacyclic code is combinatorially equivalent to an abelian code.
\end{thrm}

 We shall construct some non-central ideals in a rather simple way and prove that these are not equivalent to any abelian code.

Recall, from Theorem \ref{1.6}, that  the set of primitive central idempotents of $\F D$ is

\begin{center}
$\{e_{11}, e_{22}\}\cup \{e_{j}, \,\, 1\leq j\leq m \}.$ 
\end{center}

Let us denote by $e$, for short,  one  idempotent chosen from the set $\{e_{j}, \,\, 1\leq j\leq m \}$ which we fix for the rest of this seection. It was shown in \cite[Proposition 2.1]{JLP} that $(a-a^{-1})e$ is invertible in $\F D$ and that the elements

\newpage
$$ e_{11} = \left(\frac{1+b}{2}\right) e, \hspace{2.5cm} \ \  \ \  e_{12} = \left(\frac{1+b}{2}\right)a\left(\frac{1-b}{2}\right) e,   $$
$$e_{21} = 4((a-a^{-1})e)^{-2}\left(\frac{1-b}{2}\right) a\left(\frac{1+b}{2}\right)e, \ \  e_{22}= \left(\frac{1-b}{2}\right)e. $$
 form a set of matrix units for $(\F D)e$; i.e., they verify the identities $e_{11} + e_{22} = e$ and $e_{ij}e_{hk} = \delta_{jh}e_{1k}$. Hence, by \cite[Lemma VI.3.11]{sehgal}, $(\F D)e$ is the full ring of $2\times 2$ matrices over the centralizer of $e_{12}$ in $(\F D)e$ and thus, $e_{11}$ and $e_{22}$ are primitive (non-central) idempotents.

Notice that, setting $H^*_j = \langle b \rangle . H_j, \ 1\leq j \leq m,$ we have that \linebreak  $e_{11} = \widehat{H^*_j}-\widehat{H^*_{j-1}}$ so Lemma~\ref{1.4} readily gives  
$$dim[ (\F G)e_{11}] = \varphi(p^j) \ \ \mbox{ and } \ \ w[(\F G)e_{11}] = 2|H_j| = 4p^{m-j}.$$

It is not difficult to show that $(\F G)e_{22}$ has the same dimension and weight.

 It turns out that these particular idempotents generate codes $I_{11} = (\F D)e_{11}$ and $I_{22}=(\F D)e_{22}$ that are  equivalent to cyclic codes. However, it will be easy to use them, under certain hypotheses, to produce better codes. 
 
 To see this,  consider first   a  cyclic group $G$ of order $2p^{m}$, which we write as  $G= \tilde{A}\times C$, where $\tilde{A}$ is the cyclic subgroup of $G$ of order $p^{m}$ and $C=\{1,t\}$  its  subgroup of order $2$. Denote by 
$$\tilde{A}=K_{0}\supseteq K_{1}\supseteq \cdots \supseteq K_{m}=\{1\}$$ 
 the descending chain of   subgroups of $\tilde{A}$ and, as before, write $\tilde{e}_j = \widehat{K_j} = \widehat{K_{j-1}}$. 
 
\begin{prop} \label{1.13} With the notations above, the map

\begin{center}
$\gamma: \, D\longrightarrow G=A\times C_{2}  $\\
$\gamma(a^{i}b^{j})\longmapsto\tilde{a}^{i}t^{j}$
\end{center}
extended linealy to $\ D$ is such that $\gamma(e_{11}) = \left(\frac{1+t}{2}\right)\tilde{e}_j$ and $\gamma(e_{22}) = \left(\frac{1-t}{2}\right)\tilde{e}_j$ so

$$ \gamma(I_{j_{1}})=\F_{q}G \left(\frac{1+t}{2}\right)\tilde{e}_j  \ \ \mbox{ and } \ \ \gamma(I_{j_{2}})=\F_{q}G  \left(\frac{1-t}{2}\right)\tilde{e}_j.$$  

Consequently, $I_{11}$ and $I_{22}$ are combinatorially equivalent to cyclic codes.
 \end{prop}

\begin{proof} It follows directly, from the definition of $\gamma$ that  $\gamma(\widehat{H_i})= \widehat{K_i}$ and $\gamma(\widehat b{H_i})= \widehat t{K_i}$. We claim that
$$\gamma(g e_{11}) = \gamma(g)  \left(\frac{1+t}{2}\right)\tilde{e}_j, \ \ \mbox{ for all } g \in D.$$

In fact, $g$ is of the form $g = a^hb$, for some positive integer $h$ and, since $b(1+b)/2 = (1+b)/2$, we have 

\begin{eqnarray*} \gamma(g e_{11}) &  = & \gamma\left(a^hb\left(\frac{1+b}{2}\right) e\right) = \gamma\left( \left(\frac{a^h+a^hb}{2}\right)(\widehat{H_j} - \widehat{H_{j-1}})\right) \\
               &  = & \left(\frac{\tilde{a}^h+\tilde{a}^ht}{2}\right)(\widehat{K_j} - \widehat{K_{j-1}})  = \tilde{a}^h\tilde{b} \;\tilde{e}_j = \gamma(g)\tilde{e}. 
               \end{eqnarray*}

Also, since $b(1-b)/2 = -(1-b)/2$, in a similar way we obtain that $\gamma(ge_{22}) = \gamma(g)\tilde{e}_{22}$. 

Since $\gamma$ is linear, the   results   follow.
 \end{proof}

Set $\alpha =e _{11}+e _{12}+e_{22}$. Then, as $\alpha (e_{11}-e_{12}+e_{22}) = e$ it follows that $\alpha$is invertible in the component $(\F G)e$ with $\alpha^{-1} =e_{11}-e_{12}+e_{22}$. Then $\alpha e_{11} \alpha ^{-1} =e_{11}-e_{12}$, a non central idempotent.

  Since   conjugation is an $\F$-automorphism of $(\F D)e$, the dimension of the left ideal $I = \F_{q}D(e_{11}-e_{12})$ is also equal to $\varphi(p^{j})$.

\begin{prop} \label{NB}
 Write $f=e_{11}-e_{12}$. Then the set 

\begin{center}
$ \mathcal{B}=\{f, af, a^{2}f, \cdots, a^{\varphi(p^{j})-1}f\}$
\end{center}
is a basis for $I$ over $\F_{q}$.

\end{prop}

\begin{proof}
  We   first prove that the set $\mathcal{B}$ is linearly indenpendent.

Notice that

\begin{eqnarray*}
f & = & e_{11}-e_{12}=\left(\frac{1+b}{2}\right)e -\left(\frac{1+b}{2}\right)a\left(\frac{1-b}{2}\right)e\\
& = & \left(\frac{1+b}{2}\right)\left(1-a\left(\frac{1-b}{2}\right)\right)e = \left(\frac{1+b}{2}\right)\left(\frac{2-a+ab}{2}\right)e\\
& = & \frac{1}{4}\left[(2-a+a^{-1})+(2+a-a^{-1})b\right].
\end{eqnarray*}

Assume that 

$$\alpha_{0}f+\alpha_{1}af+ \ldots + \alpha_{\varphi(p^{j})-1}a^{\varphi{(p^{j})}-1}f=0.$$

Then

$$\left[\left(\displaystyle\sum_{k=0}^{\varphi(p^{j})-1}\alpha_{k}a^{k}\right)(2-a+a^{-1})+ \left(\displaystyle\sum_{k=0}^{\varphi(p^{j})-1}\alpha_{k}a^{k}\right)(2+a-a^{-1})b\right]e=0.$$
and thus, in particular, 

$$\left[\left(\displaystyle\sum_{k=0}^{\varphi(p^{j})-1}\alpha_{k}a^{k}\right)(2-a+a^{-1})\right]e=0,$$
 $\F_{q}\left\langle a\right\rangle$. This implies that

$$ \left( \sum_{k=0}^{\varphi(p^{j})-1}\alpha_{k}a^{k}\right)(2-a+a^{-1})e  
=  \left[\left(\sum_{k=0}^{\varphi(p^{j})-1}\alpha_{k}a^{k}\right)e\right]\left[(2-a+a^{-1})e\right]=0.$$

Since $e $ is a primitive idempotent of $\F_{q}\left\langle a\right\rangle$, the ideal $\F_{q}\left\langle a\right\rangle e  $ is a field. 

If $j\geq 2$, the elements 2$e$, $ae$ and $a^{-1}e$ have disjoint support so the element $(2-a+a^{-1})e$ cannot be zero. 

If $j=1$, then   $(2-a+a^{-1})e_{1}= 2\widehat{H_{1}}-a\widehat{H_{1}}+a^{-1}\widehat{H_{1}}-\widehat{H_{0}}$. If $(2-a+a^{-1})e_{1}=0$ so $2\widehat{H_{1}}-a\widehat{H_{1}}+a^{-1}\widehat{H_{1}}=\widehat{H_{0}}$ a contradiction. 

So $\displaystyle\sum_{k=0}^{\varphi(p^{j})-1}\alpha_{k}a^{k}e_{j}=0$ and, by {\rm{\cite[Theorem 4.9]{AP}}}, $\alpha_{k}=0$, for all $k$.

Since the number of elements of $\mathcal{B}$ is $\varphi(p^{j})$ and the dimension of $I$ is also $\varphi(p^{j})$, the result follows.   
\end{proof}

\section{An Example}

\vr

 Let $D_{9}$ be dihedral group of order 18, set $e = e_1 = H_1-h_0$, $f = e_{11}-e_{22}$ and let us consider the ideal $I = \F_{q}D_{9}f$.

 By Proposition \ref{NB}, the set $\{f, af\}$ is a basis of this ideal. We have that

 $$ae^{(1)}   =   (2a-a^{2}+1)e_{1}+(2a+a^{2}-1)e_{1}b $$
 where
 \begin{eqnarray*}
(2-a+a^{-1})(\widehat{H_{1}}-\widehat{H_{0}}) & = & 2\widehat{H_{1}}-a\widehat{H_{1}}+a^{-1}\widehat{H_{1}}-2\widehat{H_{0}}\\ a(2-a+a^{-1})(\widehat{H_{1}}-\widehat{H_{0}}) & = & 2a\widehat{H_{1}}-a^{2}\widehat{H_{1}}+\widehat{H_{1}}-2\widehat{H_{0}},
\end{eqnarray*}
 consequently  we can write an arbitrary element $\alpha$ of $\F_{q}D_{9}f$ as

\begin{eqnarray*} \alpha_{0}f+\alpha_{1}af & = & (2\alpha_{0}+\alpha_{1})\widehat{H_{1}}+(-\alpha_{0}+2\alpha_{1})a\widehat{H_{1}}+(\alpha_{0}-\alpha_{1})a^{2}\widehat{H_{1}}+ \\
  & &  (-2\alpha_{0}-2\alpha_{1})\widehat{H_{0}}+(2\alpha_{0}-\alpha_{1})\widehat{H_{1}}b+(\alpha_{0}+2\alpha_{1})a\widehat{H_{1}}b+\\   & & (-\alpha_{0}+\alpha_{1})a^{2}\widehat{H_{1}}b+(-2\alpha_{0}-2\alpha_{1})\widehat{H_{0}}b.
  \end{eqnarray*}

A direct computation shows that, if the characteristic of $\F_{q}$ is different from 2,3,5 and 7, then at most just one of the coefficients of $\alpha$ can be equal to zero. So, the weight of $\alpha$ is $w(\alpha)\geq 15$. 

As it is easy to exibit elements of $I$ of weight 15, we have:

\begin{enumerate}
 	\item The dimension of $I = \F_{q}D_{9}f$ is $\varphi(3)=2$;
	\item The weight of $ I$ is $w(I)=15$.
\end{enumerate}

We remark that the weight of this code is the same as that of the best known code  of same dimension (see [\textit{www.codetables.de}]), for exemple in the case when the field is $\F_{11}$, which satisfies our conditions.\\

Finally, we shall show that, when the multiplicative order of $q$ modulo $9$ is  $\varphi(9)=6$,  this code is not equivalent to any abelian code. Let $A$ be an abelian group of order 18; then it can be written as $A=C_{2}\times B$, where $C_{2}$ is a cyclic group of order 2 generated by an element $t$, and $B$ is an abelian 3-group.

Supose, by   way of contradiction, that   $I = \F_{q}D_{9}f$ is combinatorially equivalent to a code $J$ of $\F_{q}A$. Since ${\mathcal{U}}(\Z_{p^{r}})=\left\langle \overline{q}\right\rangle$, by \cite[Lemma 5 and Theorem 4.1]{RP}, the primitive idempotents of $\F_{q}A$ are of the form $\left(\frac{1+t}{2}\right)e_{H}$, or $\left(\frac{1-t}{2}\right)e_{H}$ where $H$ is a subgroup of $B$ such that $B/H$ is cyclic of order 9 or 3. We have two possibilties:

\begin{enumerate}
	\item $J=I_{e_{H}}$.
	
	Since the dimension of $I$ is 2, there exists a subgroup $H$ of $B$ such that $B/H$ is cyclic of order 3 and $J=(\F_{q}A)(\frac{1\pm t}{2})e_{H}$. By \cite[Theorem 4.2]{RP}, the weight of $J$ is 12 , but the weight of $\F_{q}D_{9}f$ is 15.
	
	\item $J=I_{e_{H_{1}}}\oplus I_{e_{H_{2}}}$.
	
	Since the dimension of $\F_{q}D_{9}f$ is 2, the dimension of $I_{e_{H_{1}}}$ and $I_{e_{H_{2}}}$ must be equal to 1. In this case, the order of $B/H_{1}$ and $B/H_{2}$ should be 2, which cannot happen since $|B|=9$.
\end{enumerate}

\end{document}